\documentclass{amsproc}
\usepackage[margin=1in]{geometry}
\usepackage{amssymb,amsfonts,latexsym,amsmath,amsthm,graphicx}
\usepackage{parskip}
\usepackage{mathrsfs}
\usepackage{mathtools} 
\mathtoolsset{showonlyrefs,showmanualtags}

\usepackage[margin=1in]{geometry}

\makeatletter
\def\thm@space@setup{%
  \thm@preskip=\parskip \thm@postskip=0pt
}
\makeatother

\numberwithin{equation}{section}
\newtheorem{thm}{Theorem}[section]
\newtheorem{prop}[thm]{Proposition}
\newtheorem{cor}[thm]{Corollary}
\newtheorem{lemma}[thm]{Lemma}
\newtheorem{conj}[thm]{Conjecture}
\newtheorem{prob}[thm]{Problem}
\theoremstyle{remark} 
\newtheorem{remark}[]{Remark}

\newcommand{\CC}{\mathbb{C}}

\newcommand{\hr}{\hat{r}}
\newcommand{\hC}{\hat{\CC}}

\usepackage[colorinlistoftodos]{todonotes}

\newcommand{\ol}{\overline}

\newcommand{\be}{\begin{equation}}
\newcommand{\ee}{\end{equation}}
\newcommand{\bt}{\begin{theorem}}
\newcommand{\et}{\end{theorem}}

\begin{document}

\title[Valence of logharmonic polynomials]{On the valence of logharmonic polynomials}
\date{}

\author[D. Khavinson]{Dmitry Khavinson}
\address{Department of Mathematics and Statistics, University of South Florida, Tampa, FL 33620}
\email{dkhavins@usf.edu}
\thanks{The first-named author acknowledges support from the Simons Foundation (grant 513381).}

\author[E. Lundberg]{Erik Lundberg}
\address{Department of Mathematical Sciences,
Florida Atlantic University, Boca Raton, FL 33431}
\email{elundber@fau.edu}
\thanks{The second-named author acknowledges support from the Simons Foundation (grant 712397).}

\author[S. Perry]{Sean Perry}
\address{Department of Mathematics and Statistics, University of South Florida, Tampa, FL 33620}
\email{perry2@usf.edu}

\begin{abstract}
Investigating a problem posed by W. Hengartner (2000), we study the maximal valence (number of preimages of a prescribed point in the complex plane) of logharmonic polynomials, i.e., complex functions that take the form $f(z) = p(z) \overline{q(z)}$ of a product of an analytic polynomial $p(z)$ of degree $n$ and the complex conjugate of another analytic polynomial $q(z)$ of degree $m$.
In the case $m=1$, we adapt an indirect technique utilizing anti-holomorphic dynamics to show that the valence is at most $3n-1$. 
 This confirms a conjecture of Bshouty and Hengartner (2000). Using a purely algebraic method based on Sylvester resultants, we also prove a general upper bound for the valence showing that for each $n,m \geq 1$ the valence is at most $n^2+m^2$.  This improves, for every choice of $n,m \geq 1$, the previously established upper bound $(n+m)^2$ based on Bezout's theorem.
 We also consider the more general setting of polyanalytic polynomials where we show that this latter result can be extended under a nondegeneracy assumption.
\end{abstract}

\subjclass[2020]{30C55, 37F10, 13P15}

\maketitle

\section{Introduction}

The following problem was posed by W. Hengartner at the second international workshop on planar harmonic mappings at the Technion, Haifa, January 7-13, 2000.  The problem has been restated several times (see \cite{Problems}, \cite{BshoutyHengartner}, and \cite{AbdulhadiHengartner}).

\begin{prob}[Hengartner's Valence Problem]
\label{prob:Hen}
Let $f(z) = p(z) \ol{q(z)}$ be a logharmonic polynomial of degree $n = \deg p$ and $m = \deg q$ with $p$ not a constant multiple of $q$.  Find a sharp upper bound for the valence (number of preimages of a prescribed $w\in\CC$) of $f$, i.e., determine the maximal valence in terms of $n$ and $m$.
\end{prob}

The condition that $p$ is not a constant multiple of $q$ ensures that the valence is finite, see \cite[Sec. 3]{AbdulhadiHengartner}.
The term \textit{logharmonic} refers to the fact, which is apparent, that the log of such an $f$ is locally harmonic on the set excluding its zeros. 
Concerning the case $m=1$, Bshouty and Hengartner made the following conjecture in \cite{BshoutyHengartner}.

\begin{conj}[Bshouty, Hengartner, 2000]
In the case $m=1$, the maximal valence is $3n-1$.
\end{conj}

In referring to the maximal valence, this conjecture implicitly contains two assertions: (i) in the case $m=1$, for each $n \geq 1$ the valence is at most $3n-1$ and (ii) for each $n \geq 1$ there exists a polynomial $p(z)$ of degree $n$, a linear $q(z)$, and a $w \in \CC$ such that there are $3n-1$ preimages of $w$ under $f$.
In this paper, as our first main result, we prove the following theorem confirming the first assertion (i), i.e., the upper bound that is entailed in Bshouty and Hengartner's conjecture.

\begin{thm}\label{thm:BHconjUB}
Let $p$ be a polynomial of degree $n>1$ and let $q(z)$ be linear.  For each $w \in \CC$, the number of solutions of the equation $p(z)\ol{q(z)}=w$ is at most $3n-1$.
\end{thm}

Regarding the assertion (ii), i.e., the sharpness part of the conjecture, examples are provided in \cite{BshoutyHengartner} having valence $3n-3$, which shows that the above result is asymptotically sharp.  There is additionally an example presented in \cite{BshoutyHengartner} for $n=2$ with valence $3n-1=5$ showing sharpness in the case $n=2$. In Section \ref{sec:numerics} below, we present additional examples that, according to our numerical simulations, attain the upper bound $3n-1$ in the case $n=3,4$.  It is an enticing open problem to confirm the second statement (ii) implicit in Bshouty and Hengartner's conjecture by producing examples showing that Theorem \ref{thm:BHconjUB} is sharp for each $n$.

Returning to the general setting of Hengartner's Problem,
it was observed in \cite{AbdulhadiHengartner} that Bezout's Theorem gives an upper bound $(n+m)^2$ on the valence, and the authors conjectured the following.

\begin{conj}[Abdulhadi, Hengartner, 2001]
The Bezout bound fails to be sharp for every $m,n \geq 1$, i.e., the maximal valence is strictly less than $(m+n)^2$.
\end{conj}

Since $n^2+m^2 < (m+n)^2$ when $m,n \geq 1$, the following theorem, our second main result of this paper, improves Bezout's bound for every $m,n \geq 1$, hence confirming the above conjecture of Abdulhadi and Hengartner.

\begin{thm}\label{thm:general}
Let $f(z) = p(z) \ol{q(z)}$ be a logharmonic polynomial with $n := \deg p \geq 1$, $ m := \deg q \geq 1$, and $p$ not a constant multiple of $q$.  Then the valence of $f$ is at most $n^2 + m^2$.
\end{thm}

This result admits further improvement, at least in the case $m=1$, in light of Theorem \ref{thm:BHconjUB}.  The question was raised in \cite{BshoutyHengartner} whether the maximal valence is actually linear in $m,n$. We suspect that the maximal valence grows linearly in $n$ for each fixed $m$ but grows quadratically when $n=m$ or when $n$ is asymptotically proportional to $m$.

Our proof of Theorem \ref{thm:general} uses a purely algebraic technique based on polynomial resultants applied to the system of equations obtained by considering $p(z) \ol{q(z)} - 1 =0$ and its complex conjugate $\ol{p(z)} q(z) - 1 = 0$.  
A technical obstacle in using this method is establishing the following result (stated as Lemma \ref{lem:coprime1} in Section \ref{sec:general}) showing that the left hand sides of these equations, viewed as bivariate polynomials in $z, \ol{z}$, are coprime.

\begin{lemma}\label{lem:coprimeIntro}
Given two complex analytic univariate polynomials $p$ and $q$ with $p$ not a constant multiple of $q$, the polynomials $P(z,\overline{z}) = p(z)\overline{q(z)} - 1 $ and $Q(z,\ol{z}) = \overline{p(z)}q(z) - 1$, viewed as bivariate polynomials in $z, \ol{z}$ are coprime.
\end{lemma}

As a corollary of Lemma \ref{lem:coprimeIntro} we recover the following aforementioned result that was proved in \cite[Sec. 3]{AbdulhadiHengartner} using several complex analytic tools (including the Riemann mapping theorem, Schwarz reflection, analytic continuation, and properties of Blaschke products), whereas our proof of Lemma \ref{lem:coprimeIntro} is purely algebraic and essentially relies on a single property of resultants (namely, that coprimacy is equivalent to the resultant not vanishing identically).

\begin{cor}
If $p$ is not a constant multiple of $q$, then the logharmonic polynomial $p(z) \ol{q(z)}$ has finite valence.
\end{cor}

Concerning the \emph{minimal} valence we note, as was observed in \cite[Thm. 3.6]{AbdulhadiHengartner}, that for $n>m$ it follows from a generalization of the argument principle that the valence is bounded below by $n-m$. This lower bound is sharp (and hence gives the minimal valence) as it is easy to check (especially in polar coordinates $z = r e^{i \theta}$) that the equation $z^n \ol{z}^m = 1$ has exactly $n-m$ solutions (located at roots of unity).

\subsection{Sheil-Small's valence problem for harmonic polynomials}

Hengartner's valence problem for logharmonic polynomials is reminiscent of another valence problem posed by T. Sheil-Small, where instead of a product of $p(z)$ and $\ol{q(z)}$ a sum $p(z) + \ol{q(z)}$ is considered.

\begin{prob}[Sheil-Small's Valence Problem]
\label{prob:SS}
Determine the maximal valence of harmonic polynomials $p(z) + \ol{q(z)}$ with $n = \deg p > m = \deg q$.
\end{prob}

In his thesis \cite{W1}, A. S. Wilmshurst used the maximum principle for harmonic functions together with Bezout's theorem to show that the valence is at most $n^2$, and he constructed $n^2$-valent examples with $m=n-1$.  This confirmed a conjecture of Sheil-Small that $n^2$ is the maximum valence.  Wilmshurst conjectured \cite{W1} (cf. \cite[Remark 2]{W2}) that for each pair of integers $n>m\geq1$, the maximum valence is $3n-2 + m(m-1)$.  This conjecture was confirmed for $n=m-1$ by the above mentioned results of Wilmshurst and for $m=1$ by a result of the first-named author and G. Swiatek \cite{KhSw} together with L. Geyer's proof \cite{G2008} of the Crofoot-Sarason conjecture.
However, counterexamples to Wilmshurst's conjecture for the case $m=n-3$ were provided in  \cite{LLL}.
Further counterexamples were provided in \cite{LS}, \cite{HLLM}, and \cite{Random}. In spite of these counterexamples, it is still seems likely that, in the spirit of Wilmshurst's conjecture, the maximal valence increases linearly in $n$ for each fixed $m$ \cite{LLL}, \cite{KhavLeeSaez}, \cite{Random}.
The latter result \cite{Random} used a nonconstructive probabilistic method, thus connecting the study of the extremal problem with a parallel line of research on the average number of zeros of random harmonic polynomials that had been investigated (for a variety of models) in \cite{LiWei}, \cite{Lerariotruncated},
\cite{Andy}, \cite{AndyZach}.

Harmonic polynomials and logharmonic polynomials are just two species of so-called polyanalytic polynomials, that is polynomials $P(z,\ol{z})$ in $z$ and $\ol{z}$ that satisfy $\left(\frac{\partial}{\partial \ol{z}}\right)^{m+1} P(z,\ol{z}) = 0$ with $\frac{\partial}{\partial \ol{z}}:=\frac{1}{2} \left( \frac{\partial}{\partial x} + i \frac{\partial}{\partial y} \right)$ and hence take the form
$$ P(z,\ol{z}) = \sum_{k=0}^m p_k(z) \ol{z}^k,$$
where $p_k(z)$ are complex-analytic polynomials.
The above Problems \ref{prob:Hen} and \ref{prob:SS} thus fit into a broader setting concerning the valence of polyanalytic polynomials.  Polyanalytic polynomials have been studied classically \cite{Balk} and have also arisen in more recent studies in gravitational lensing \cite{Petters}, \cite{Perry}, approximation theory \cite{approx}, \cite{UnifApprox}, \cite{KhavBianalytic} numerical linear algebra \cite{Ginibre}, random matrix theory \cite{Huhtanen}, and determinantal point processes \cite{Abreu}.
Motivated by the valence problems for harmonic and logharmonic polynomials as well as the prominent role played by zero sets in several of the aforementioned studies on polyanalytic polynomials, it seems natural to pose the following general valence problem for polyanalytic polynomials.

\begin{prob}
Determine conditions on the coefficient polynomials $p_k(z)$ which guarantee that $P(z,\ol{z})$ has finite valence, and estimate the maximal valence in terms of $m$ along with the degrees of $p_k(z)$.
\end{prob}

The proof of Theorem \ref{thm:general} in Section \ref{sec:general} leads to the following result (see Lemma \ref{lem:PRL} in Section \ref{sec:general}) estimating the valence of polyanalytic polynomials for which $\deg p_k \leq n$ for each $k=0,1,...,m$.

\begin{prop}
Let $ P(z,\ol{z}) = \sum_{k=0}^m p_k(z) \ol{z}^k$ be a polyanalytic polynomial with $\deg p_k \leq n$ for each $k=0,1,...,m$.  Assume $P(z,\ol{z})$ and $Q(z,\ol{z}) = \ol{P(z,\ol{z})}$ are coprime.  Then the valence is at most $n^2+m^2$. 
\end{prop}

\subsection{Proof techniques}
The proof of Theorem \ref{thm:BHconjUB} begins with a reformulation of the valence problem in terms of rational harmonic functions and then specializes a technique previously used in that setting by the first named author and G. Neumann in \cite{KN}, where the key idea is to use a theorem of Fatou from holomorphic dynamics.
This indirect method was introduced by the first named author and G. Swiatek in \cite{KhSw} in their above-mentioned proof of the $m=1$ case of Wilmshurst's conjecture.  Subsequent adaptations of this method have led to solutions to additional problems including sharp estimates for the topology of quadrature domains \cite{LeeMakarov}, classification of the number of critical points of Green's function on a torus \cite{BergErem}, and sharp estimates for the number of solutions of certain transcendental equations \cite{BergErem2018}.
The results of \cite{KN}, which are directly relevant to the current paper, confirmed astronomer S. Rhie's conjecture in gravitational lensing (motivated by the connection to gravitational lensing the zeros of rational harmonic functions were investigated further in \cite{Bleher}, \cite{SeteCMFT}, \cite{SeteGrav}, \cite{SetePert}, \cite{Zur2018a}, \cite{Zur2018b}).

The proof of Theorem \ref{thm:general} utilizes another technique that has been applied in studies of gravitational lensing, in this instance, going in a separate, purely algebraic direction utilizing Sylvester resultants.  
The application of Sylvester resultants to bound the valence of certain polynomials in $z$ and its complex conjugate $\ol{z}$ was introduced by A.O. Petters in the study of gravitational lensing in \cite{Petters}.  Our proof of Theorem \ref{thm:general} particularly resembles the more recent application of resultants in the study of gravitational lensing by multi-plane point mass ensembles carried out by the third-named author in \cite{Perry}.


\subsection{Outline of the paper}
We present the proof of Theorem \ref{thm:BHconjUB} in Section \ref{sec:BHconjUB} where we begin by reviewing some preliminary results in Section \ref{sec:prelim} before proceeding to the proof in Section \ref{sec:proofBHconjUB}. Our proof of Theorem \ref{thm:BHconjUB} adapts the method from \cite{KN} with an important modification, see Remark \ref{rmk:modification}. We discuss Bshouty and Hengartner's  $(3n-3)$-valent examples from the perspective of holomorphic dynamics in Section \ref{sec:numerics}, where we also present the results of some numerical experiments.  We present the proof of Theorem \ref{thm:general} in Section \ref{sec:general}, where we include review of some essential algebraic preliminaries related to Sylvester resultants.


\section{Anti-holomorphic dynamics and the case $m=1$}\label{sec:BHconjUB}

\subsection{Preliminaries}\label{sec:prelim}

In this section, we collect some notation and terminology along with several lemmas that will be used in the proof of Theorem \ref{thm:BHconjUB}.



\subsubsection{Anti-holomorphic dynamics}
We will need the following result that adapts a classical theorem of Fatou from the setting of holomorphic dynamics to the setting of anti-holomorphic dynamics.

A point $z_0 \in \hC$ is a \emph{critical point} of $r$ if the spherical derivative of $r$ vanishes at $z_0$.
The complex conjugate of a rational function is referred to as an \emph{anti-rational map}.  A fixed point $z_0$ of the anti-rational map $z \mapsto \ol{r(z)}$ is referred to as an \emph{attracting} fixed point if $|r'(z_0)|<1$.
A fixed point $z_0$ is said to attract some point $w \in \hat{\CC}$ if the
iterates of $z \mapsto \ol{r(z)}$ starting at $w$ converge to $z_0$.

The following result is taken from \cite[Thm. 2.6]{GeyerRational}.

\begin{lemma}[Fatou's Theorem for anti-rational maps]\label{lem:Fatou}
Let $r$ be a rational function with degree greater than $1$. If $z_0$ is an attracting fixed point of the anti-rational map $z \mapsto \ol{r(z)}$, then $z_0$ attracts at least one critical point of $r$. 
\end{lemma}


\subsubsection{The argument principle for harmonic maps}

We will also need the following generalization of the argument principle for harmonic functions.
We recall that a zero $z_0$ of a harmonic function is referred to as a \emph{singular zero} of $F(z)=h(z) + \ol{g(z)}$ if the Jacobian of $F$, that is easily computed to be $|h'(z)|^2 - |g'(z)|^2$, vanishes at $z_0$. Moreover, $F$ is said to be sense-preserving (respectively sense-reversing) at $z_0$ if the Jacobian is positive (respectively negative) at $z_0$.  The  \emph{order} of a sense-preserving zero is defined to be the smallest positive integer $k$ such that $h^{(k)}(z_0) \neq 0$.  The order of a sense-reversing zero $z_0$ is defined to be $-k$ where $k$ is the order of $z_0$ as a sense-preserving zero of $\ol{F(z)}$.
If $F$ is harmonic in a punctured neighborhood of $z_0$ and $F \rightarrow \infty$ as $z \rightarrow z_0$ then $z_0$ is referred to as a \emph{pole}. 
As defined in \cite{ST}, the \emph{order of a pole} is given by $\frac{1}{2\pi} \Delta_{T} \arg F$, where $T$ is a sufficiently small circle centered at the pole, and $\Delta_{T} \arg F$ denotes the increment of the argument of $F(z)$ along $T$.   It follows that the order of a pole $z_0$ is negative when $F$ is sense-reversing in a punctured neighborhood of $z_0$, and elaborating on this point we set aside the following remark that will be used later.

\begin{remark}\label{rmk:pole}
In the sequel, we will consider the harmonic rational function
$\frac{1}{\ol{p(z)}} - b - z$.  We notice that this function has a pole at each zero of $p$ and is sense-reversing in a neighborhood of each pole.  Hence, the orders of the poles are all negative, and the sum of the orders is $-n$.
\end{remark}

The following version of the argument principle is taken from \cite[Thm. 2.2]{ST}.  We note that there are topological versions of the argument principle (having classical roots going back to the Poincar\'e-Hopf Theorem for vector fields) that hold in more general settings \cite{sheil-small_2002}, \cite{Cristea}.

\begin{lemma}[The Generalized Argument Principle for harmonic maps]\label{lem:argprinc}
Let $F$ be harmonic, except for finitely many poles, in a region $D$, and let $C$ be a Jordan curve that does not pass through any zeros or poles of $F$.  Let $\Omega$ denote the interior of $C$ and assume $F$ has no singular zeros in $\Omega$.  Let $\Delta_C \arg F$ denote the increment of the argument of $F(z)$ as $z$ traverses $C$. Let $n_+$, $-n_-$ be the sum of the orders of the sense-preserving and sense-reversing (respectively) zeros of $F$ in $\Omega$ and let $p_+, -p_-$ be the sum of the orders of the sense preserving and sense reversing (respectively) poles of $F$ in $\Omega$. Then 
\be
\frac{1}{2\pi}\Delta_{C} \arg F= n_+ - n_- - (p_+ - p_-).
\ee
\end{lemma}

\subsubsection{Tools for reducing the general case to the generic case}

The generalized argument principle stated in Lemma \ref{lem:argprinc} above can only be applied in the nondegenerate case where the harmonic mapping is free of singular zeros.
The next two lemmas will be used to reduce the problem of establishing the upper bound in the general case (where singular zeros are allowed) to establishing it for the nondegenerate case. Harmonicity comes into play in this  reduction, and there is no analogous strategy in general smooth settings.  To elaborate, the first of these lemmas establishes a genericity result for nondegenerate examples.  It is an important principle of differential topology that results of this type hold more generally for smooth maps. On the other hand, the second lemma establishes a stability result for extremal examples, and as pointed out in \cite{BergErem} this type of result does not hold for arbitrary smooth maps.

The following result is taken from \cite[Lemma 1]{KN}.

\begin{lemma}[nondegenerate examples are generic ]\label{lem:perturb}
If $r(z)$ is a rational function of degree greater than 1, then the set of complex numbers $c$ for which $\ol{r(z)} - z - c$ does not have singular zeros is open and dense in $\CC$.
\end{lemma}

The following result is taken from \cite[Prop. 3]{BergErem}.

\begin{lemma}[extremal examples are stable]\label{lem:open}
Let $f:D \rightarrow \CC$ be a harmonic function defined in a region $D$ in $\CC$.
Suppose that every $w \in \CC$ has at most $\mu$ preimages, where $\mu < \infty$. Then the
set of points which have $\mu$ preimages is open.
\end{lemma}

\subsection{Proof of Theorem \ref{thm:BHconjUB}}\label{sec:proofBHconjUB}

The problem is to bound the number of solutions of 
\be\label{eq:original}
p(z) \ol{q(z)} = w
\ee
when $q(z)$ is of degree $m=1$ and $p$ is a polynomial of degree $n>1$. 
By factoring out the leading coefficient of $q$ and assimilating it into $p$, we may assume without loss of generality that $q(z) = z+b$, $b\in\CC$.

Since we easily get an upper bound of $n+1$ in the case $w=0$ (by setting each factor $p(z)$ and $\ol{z+b}$ to zero), we may assume $w \neq 0$.  By multiplying $p$ by $1/w$ we may further assume without loss of generality $w=1$.
Applying these reductions and rearranging \eqref{eq:original}, the problem is then to bound the number of solutions of the equation
\be\label{eq:reduced}
\frac{1}{\ol{p(z)}} - b - z = 0,
\ee
i.e., zeros of the harmonic rational function $z \mapsto \frac{1}{\ol{p(z)}} - b - z$.
We first estimate the number of sense-reversing zeros
in the following proposition that is a specialized version of \cite[Prop. 1]{KN}.  The proof of the proposition essentially follows \cite{KN} with one important modification (see Remark \ref{rmk:modification} below).  We simplify the presentation of the proof (in comparison with the proof of \cite[Prop. 1]{KN}) by making use of Lemma \ref{lem:Fatou} in place of the classical theorem of Fatou and by settling for a slightly weaker (yet sufficient for our purposes) result that avoids consideration of so-called ``neutral'' fixed points.

\begin{prop}\label{prop:dynamics}
Let $p$ be a polynomial of degree $n \geq 2$.
Then the number of sense-preserving zeros of 
$\displaystyle \frac{1}{\ol{p(z)}} - b -z$ is at most $n$.
\end{prop}

\begin{proof}
Let $r(z) = \frac{1}{p(z)}-\ol{b}$.
The sense-preserving zeros of $\displaystyle \frac{1}{\ol{p(z)}}-b-z$ are attracting fixed points of the anti-rational map $z \mapsto \ol{r(z)}$.
By Lemma \ref{lem:Fatou}, the number of such attracting fixed points is at most the number of critical points of $r$, which is a rational function of degree $n$.  By the Riemann-Hurwitz relation \cite[Sec. 17.14]{Forster}, a rational function of degree $n$ has exactly $2n-2$ critical points (counted with multiplicity), but the special case $r(z) = \displaystyle \frac{1}{p(z)} - \ol{b}$ of the shifted reciprocal of a polynomial of degree $n$ has a critical point of order $n-1$ at infinity along with $n-1$ finite critical points (located at the zeros of $p$).  Consequently,
the anti-rational map $z \mapsto \ol{r(z)} $ has at most $n-1+1=n$ attracting fixed points, i.e., the harmonic rational function $\ol{r(z)}-z = \displaystyle \frac{1}{\ol{p(z)}}-b-z$ has at most $n$ sense-preserving zeros, as desired.
\end{proof}

 \begin{remark}\label{rmk:modification}
 To reiterate an important point in the above proof of Proposition \ref{prop:dynamics}, we note that while being of high multiplicity (counted $n-1$ times in the Riemann-Hurwitz relation), the critical point at infinity only counts once in the dynamics argument, since it can be attracted to at most one attracting fixed point.  In comparison with general anti-rational maps, this diminishes the possible number of attracting fixed points by $n-2$ which is responsible for the ultimate improvement in the upper bound $3n-1$ appearing in Theorem \ref{thm:BHconjUB} in comparison with the (sharp) upper bound $5n-5$ \cite{KN} for the valence of harmonic rational functions of the form $\ol{r(z)}-z$ when $r$ is a general analytic rational function (the improvement $n-2$ gets doubled to $2n-4$ in the proof of Theorem \ref{thm:reduced} below at the step where the argument principle is applied).
 \end{remark}

From the discussion preceding Proposition \ref{prop:dynamics}, it follows that the problem of establishing Theorem \ref{thm:BHconjUB} reduces to proving the following result.

\begin{thm}\label{thm:reduced}
The number of solutions of \eqref{eq:reduced} is at most $3n-1$, where as before $n:= \deg p$.
\end{thm}

\begin{proof}
As in the proof of Proposition \ref{prop:dynamics} let us use the notation 
\be\label{eq:specified}
r(z) = \frac{1}{p(z)}-\ol{b}.
\ee
\emph{Step 1.}
We first consider the case that $\ol{r(z)}-z$ does not have singular zeros.
Then we apply Lemma \ref{lem:argprinc} on a disk $D_R$ of radius $R$ large enough that $D_R$ contains all zeros and poles of $\ol{r(z)}-z$.
Since $r(z) = \mathcal{O}(1)$ as $z \rightarrow \infty$, the increment of the argument of  $\ol{r(z)}-z$ along the boundary of $D_R$ is $2\pi$ for any $R$ sufficiently large. By Lemma \ref{lem:argprinc}, $$ 1 = n_+ - n_- - (p_+ - p_-).$$
As noted in Remark \ref{rmk:pole}, we have $p_+ = 0$ and $p_-=n$, so that the above becomes
\be 
n_- = n_+ + n - 1.
\ee
By Proposition \ref{prop:dynamics} we have $n_+ \leq n$ which together with the above gives the following estimate on $n_-$.
\be 
n_- \leq 2n - 1,
\ee
Therefore, the total number of zeros satisfies $n_+ + n_- \leq 3n-1$. This establishes the theorem in the case when $\ol{r(z)}-z$ has no singular zeros.

\emph{Step 2.} We extend the result from Step 1 to the general case (where singular zeros are allowed) by employing Lemmas \ref{lem:perturb} and \ref{lem:open} as follows.  Let $\mu(n)$ denote the maximal number of zeros, and suppose $\hr$ is a rational function of the specified form \eqref{eq:specified} such that
$\ol{\hr(z)}-z$ has $\mu(n)$ zeros. By Lemma \ref{lem:perturb}, there is a sequence of numbers $c_k \in \CC$ with $c_k \rightarrow 0$ such that $\ol{\hr(z)}-c_k-z$ is free of singular zeros for each $k$. By Lemma \ref{lem:open}, the set of $c$ such that $\ol{\hr(z)}-c-z$ attains $\mu(n)$ zeros is open. Hence, for $k$ large enough $\ol{\hr(z)}-c_k-z$ is both free of singular zeros and attains $\mu(n)$ zeros.
It then follows from Step 1 of the proof that $\mu(n) \leq 3n-1$, and this completes the proof of the theorem. 
\end{proof}






\section{Numerical results}\label{sec:numerics}

The example from \cite{BshoutyHengartner} with valence at least $3n-3$ is the following. Let
\be\label{eq:example}
f(z) = \left( \frac{z^{n}}{n} - z \right) \ol{z} = |z|^2 \left( \frac{z^{n-1}}{n} - 1 \right) .
\ee

As pointed out in \cite{BshoutyHengartner},
for $w=-\frac{n-1}{n}$
the roots of unity $z_k = e^{2\pi i k / (n-1)}$, ($k = 1,2,...,n-1$) are solutions of $f(
z) = w$.

Following the reformulation we used above, let us rewrite $f(z)=w$ as
\be\label{eq:example2}
\ol{r(z)} = z,
\ee
where
\be
r(z):= \frac{n w}{z^n - nz} = -\frac{n-1}{z^n-nz}.
\ee

Then the roots of unity $z_k$ are attracting fixed points of $\ol{r(z)}$.  Indeed, we have
\begin{align*}
    r(z_k) &= -\frac{n-1}{z_k (z_k^{n-1}-n)} \\
    &= -\frac{n-1}{z_k (1-n)} = \frac{1}{z_k} = \ol{z_k},
\end{align*}
where we used $|z_k| = 1$ in the last equality.  This shows that each root of unity $z_k$ is a fixed point.  To see that each such $z_k$ is attracting, we compute
$$r'(z) = (nz^{n-1}-n)\frac{n-1}{(z^{n}-nz)^2},$$
and notice that this vanishes when evaluated at a root of unity, i.e., we have $r'(z_k)=0$. Thus, they are actually super-attracting fixed points. 

Attracting fixed points of $\ol{r(z)}$ are orientation-preserving solutions of $\ol{r(z)} -z = 0$.  By the generalized argument principle, this leads to at least $3n-3$ total solutions.  

In fact, we can see from inspecting the dynamics of iteration of the map $z \mapsto \ol{r(z)}$ that there are exactly $3n-3$ for this example, since the critical point at infinity is not attracted to a fixed point. On the contrary, it is in a period-two orbit $\infty \rightarrow 0 \rightarrow \infty \rightarrow 0 \cdots$.  This diminishes the possible number of attracting fixed points by one, i.e., inspecting the proof of Proposition \ref{prop:dynamics} the upper bound stated in the proposition becomes $n-1$. 
This then diminishes the total count by two.  Indeed, following the argument in the proof of Theorem \ref{thm:reduced} using the generalized argument principle, we have $n_+ \leq n-1$, $n_- \leq 2n-2$, and $N = n_+ + n_- \leq 3n-3$, so that the valence of this example is at most $3n-3$ (and hence exactly $3n-3$).

In order to show sharpness of the upper bound in Theorem \ref{thm:BHconjUB}, we would like to modify this example to have an additional attracting fixed point that attracts the critical point at infinity, leading to a total of $3n-1$ solutions of $\ol{r(z)}=z$. According to numerics presented in the subsection below, this strategy succeeds at least for $n=3,4$, but it remains an open problem to verify this analytically and extend it to each $n$.

\subsection{Numerical examples showing sharpness for $n=3,4$}

Let us consider a modified version of Bshouty and Hengartner's example with the following slightly different choice of $r$ that includes two real  parameters $A>0, C \geq 0$ (adding $C$ prevents infinity from going to the origin under iteration, and including $A$ gives some extra flexibility while trying to tune parameters to produce an additional attracting fixed point)
$$ r(z):= \frac{A}{z\left[z^{n-1}+n\left(\frac{A}{n - 1}\right)^{(n - 1)/(n + 1)}\right]}+C.$$
When $C=0$ and $A=1$, this reduces to (a rotation of) the Bshouty-Hengartner example, with attracting fixed points at roots of $(-1)$ instead of roots of unity.  When $C=0$ and $A$ varies, there are still $n-1$ attracting fixed points on a common circle, namely, they are at the $(n-1)$th roots of $(-1)$ scaled by  $R(A,n):=\left(\frac{A}{n-1}\right)^\frac{1}{n+1}$, i.e., for each $k=1,..,n-1$, the point $z_k= R(A,n) e^{i\pi\frac{1+2k}{n-1}}$ is a critical point for $r(z)$ that is also a fixed point of $\ol{r(z)}$.

With specific choices of parameters, we used numerical evaluation of the built-in Mathematica command NestList in order to iterate the anti-rational map $z \mapsto \ol{r(z)}$ while tracking the trajectory of each critical point. 
According to these numerical experiments, when $n=3$, $A=100$, $C=0.68$, this example has three attracting fixed points and hence $8=3n-1$ solutions. The attracting fixed points are located at $ z \approx 1.257 \pm 2.069 i$ and $z \approx 2.31$.
The critical point at infinity is attracted to the latter fixed point, and the symmetric pair of critical points on the imaginary axis are each attracted to the nearest of the other two fixed points.  Similarly, we were able to locate parameter values that, according to numerical simulation, produce an example showing sharpness of Theorem \ref{thm:BHconjUB} in the case $n=4$. Namely, with $A=250$ and $C=.86$ we find four attracting fixed points and hence $11 = 3n-1$ solutions.
The basins of attraction for each of these extremal examples are depicted in Figure \ref{fig:Basins}.

\begin{figure}[h]
\centering
\includegraphics[scale=0.125]{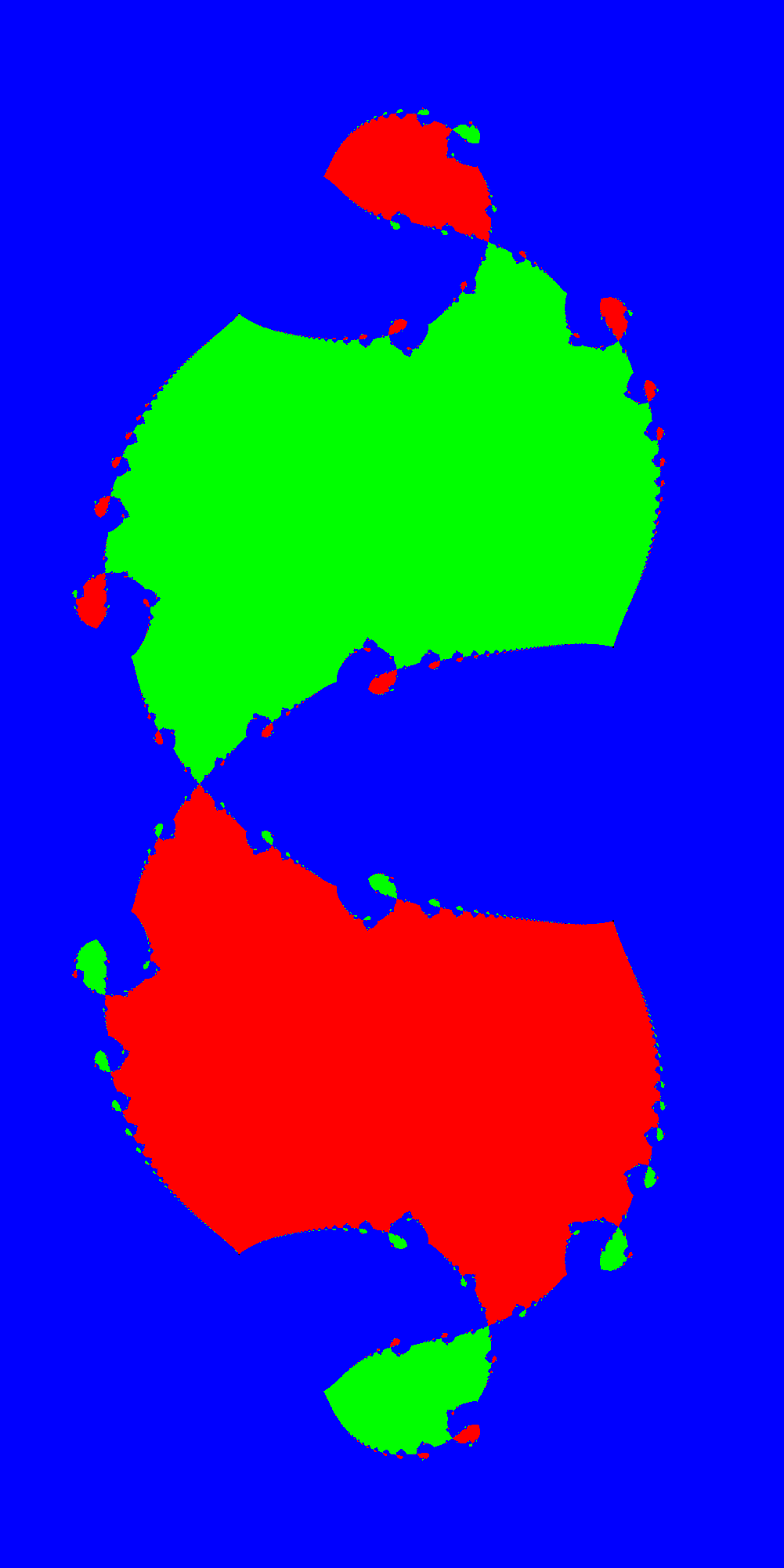}
\hspace{0.1in}
\includegraphics[scale=0.25]{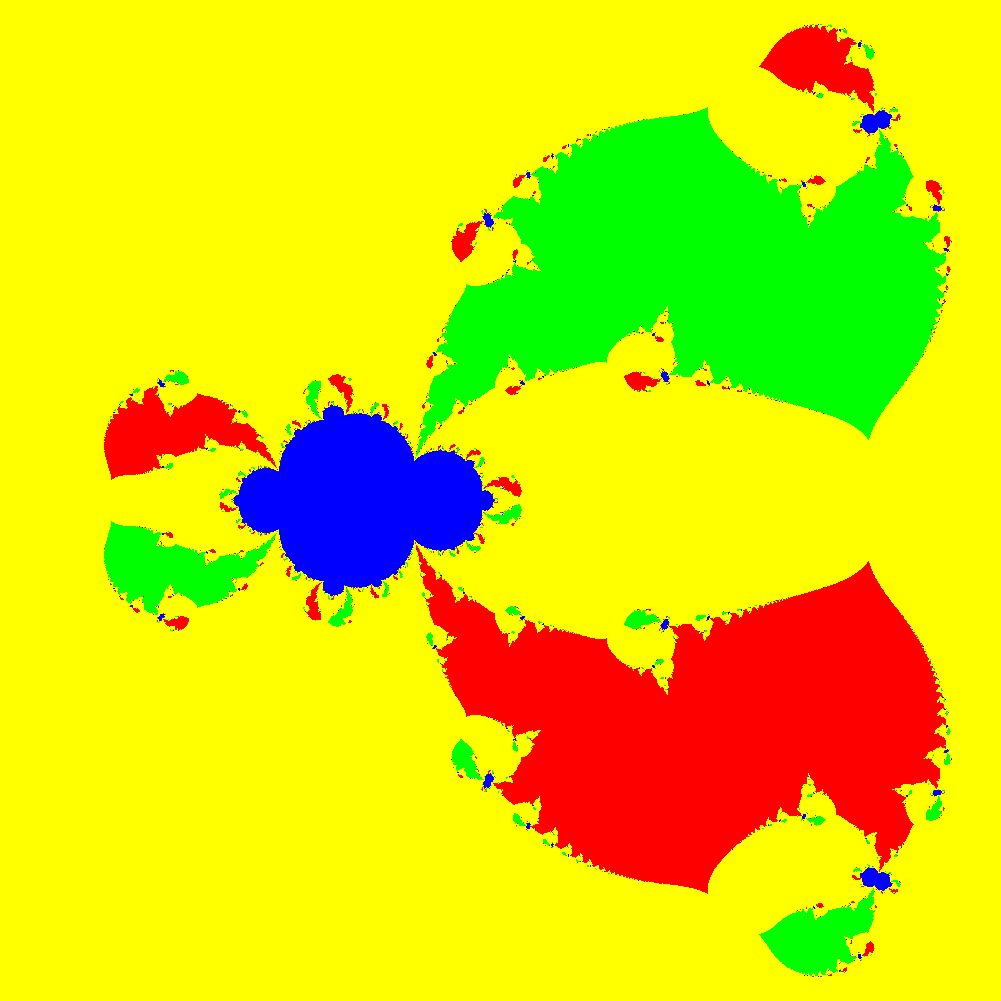}
\caption{Left: The basins of attraction for each of the three attracting fixed points of the anti-rational map $\displaystyle z \mapsto \frac{100}{\ol{z}\left[\ol{z}^{2}+3\sqrt{50}\right]}+0.68$ which is an extremal example for the case $n=3$. 
Right: The basins of attraction for each of the four attracting fixed points of the anti-rational map $\displaystyle z \mapsto \frac{250}{\ol{z}\left[\ol{z}^{3}+4\left(\frac{250}{3}\right)^{3/5}\right]}+0.86$ which is an extremal example for the case $n=4$. 
In both pictures, regions with the same color/shade are in a common basin of attraction, i.e., are attracted to a common fixed point. Image Credit: Both pictures were created by Walter Bergweiler.}
\label{fig:Basins}
\end{figure}



\section{An improvement on the Bezout bound}\label{sec:general}

Recall that Hengartner’s Valence Problem asks for a bound on the number of solutions to 
\be\label{eq:logharmonic}
p(z) \ol{q(z)} = w ,
\ee
where $w$ is an arbitrary complex number, and $p$ and $q$ are (analytic) complex polynomials, of degree $n$ and $m$ respectively, such that $p$ is not a constant multiple of $q$. 
The case $w=0$ admits a bound of $m+n$ on the valence (by separately setting $p(z)=0$, $\ol{q(z)}=0$), so we may assume $w \neq 0$.
Furthermore, without loss of generality (by multiplying by $1/w$ and renaming the coefficients of, say, $p$) we assume $w=1$.

Let us further reformulate the problem as one of bounding the number of zeros of the polynomial $P(z,\overline{z}) = p(z)\overline{q(z)}-1$.
Writing $Q(z,\ol{z}) = \ol{P(z,\ol{z})}$, we note that $P(z,\ol{z})$ and $Q(z,\ol{z})$ have the same zeros.  Hence, the number of zeros $z \in \CC$ of $P(z,\ol{z})$ is bounded from above by the number of common zeros $(z_1,z_2) \in \CC^2$ of $P(z_1,z_2)$ and $Q(z_1,z_2)$.
An application of Bezout's theorem (which can be justified when $p$ is not a constant multiple of $q$ \cite{AbdulhadiHengartner}) then gives an upper bound of $(n+m)^2$ on the number of zeros, an estimate that we will refer to as the \emph{Bezout bound}.

The main goal of this section is to improve the Bezout bound by proving Theorem \ref{thm:general}
which in light of the above discussion reduces to establishing the following result.

\begin{thm}\label{thm:genReduced}
Let $p$ and $q$ be complex polynomials of degree $n$ and $m$ respectively. If $p$ is not a constant multiple of $q$, then the number of zeros of $P(z,\ol{z})= p(z)\ol{q(z)}-1$ is bounded above by $n^2 + m^2$.
\end{thm}

Before presenting the proof of this result,
we review some preliminaries concerning Sylvester resultants.

\begin{remark}
It might seem surprising that an application of resultants leads to an improvement on the Bezout bound considering that Bezout's theorem itself is often proved using resultants \cite{SS}; the proof of Theorem \ref{thm:genReduced} indeed uses resultants in an essentially different manner, as in \cite[Sec. 11.5.1]{PetBook}, that combines fruitfully with the constraint that the second variable is taken to be the complex conjugate of the first variable, see Remark \ref{rmk:resultant} below.
\end{remark}

While our proof of Theorem \ref{thm:genReduced} uses essentially the same ideas as in \cite{Petters}, \cite{Perry}, in the current setting we are able to establish coprimacy of the polynomials $P,Q$ without imposing an additional nondegeneracy condition.  It is an open problem in the theory of gravitational lensing to determine whether the nondegeneracy assumption can be removed from the results of \cite{Petters}, \cite{Perry}, bounding the number of apparent images of a single background source lensed by a multiplane gravitational lens comprised of point masses.

\subsection{Sylvester Resultants}

Let $P\in\CC[z_1,z_2]$ be a bivariate polynomial with complex coefficients and variables. Written variously 
\be \label{eqn:bipolyP}
P(z_1, z_2) = \sum_{i=0}^n a_i(z_2)z_1^i = \sum_{j=0}^m b_j(z_1)z_2^j = \sum_{i,j} c_{i,j}z_1^iz_2^j 
\ee
we denote by $\deg_{z_1} P = n$, $\deg_{z_2} P = m$, and $\deg P = N = \max \{ i+j \;\:|\;\:c_{i,j} \neq 0 \}$ the \textit{degree in $z_1$}, the \textit{degree in $z_2$}, and the \textit{total degree} of P respectively. Of course, $\deg P \leq \deg_{z_1} P + \deg_{z_2}P$. Let $Q \in \CC[z_1,z_2]$ be another such polynomial written
\be \label{eqn:bipolyQ}
Q(z_1, z_2) = \sum_{i=0}^s d_i(z_2)z_1^i = \sum_{j=0}^t e_j(z_1)z_2^j = \sum_{i,j} f_{i,j}z_1^iz_2^j .
\ee
Denote $\deg_{z_1} Q = s$, $\deg_{z_2} Q = t$, and $\deg Q = M$. Writing $P$ and $Q$ as polynomials in $z_1$, their coefficients are polynomials in $z_2$, namely $a_0, ..., a_n$ and $d_0,...,d_s$ respectively. The \textit{$z_1$-Sylvester Matrix} $\mathcal{S}_{z_1}(P,Q)$ of $P$ and $Q$ is an $(n+s) \times (n+s)$ matrix which is constructed by diagonally stacking $s$ copies of the polynomials $a_i$ and $n$ copies of the $d_i$ with $0$ used for the other entries as follows.
\be \label{eqn:Sylvester}
\mathcal{S}_{z_1}(P,Q)=
\left[\arraycolsep=3pt\def\arraystretch{1}
\begin{array}{ccccc|cccc}
a_n & a_{n-1} & a_{n-2} & \cdots & a_1 & a_0 & 0 & \cdots & 0 \\
0 & a_{n} & a_{n-1} & a_{n-2} & \cdots & a_1 & a_0 & \cdots & 0 \\
\vdots & & \ddots &  &  &   & & \ddots \\
0 & \cdots & 0 & a_n & a_{n-1} & a_{n-2} & \cdots & a_1 & a_0 \\
 \hline 
d_s & d_{s-1} & \cdots & d_0 & 0  & 0 & \cdots  & & 0 \\
0 & d_s & d_{s-1} & \cdots &  d_0 & 0  & \cdots   & &0 \\
0 & 0 & d_s & d_{s-1} & \cdots & d_0 & \cdots & & 0 \\
\vdots &  &  & \ddots & &   & \ddots &  \\
0 &  \cdots & & 0 & d_s & d_{s-1} & \cdots &  & d_0 \\
\end{array}
\right].
\ee
Formally, for $1\leq i \leq s$, entry $\mathcal{S}_{z_1}(P,Q)_{i,j} = a_{n-j+i}$ for $i\leq j \leq i+n$ and 0 otherwise. Likewise for $s+1 \leq  i \leq s+n$ entry $\mathcal{S}_{z_1}(P,Q)_{i,j} = d_{i-j}$ for $i-s\leq j \leq i$ and 0 
otherwise.
Note that the exact location of columns in the top and bottom halves of \eqref{eqn:Sylvester} depends on the values of $n$ and $s$, but the matrix always has the indicated block structure where the top right $s \times s$ block is lower triangular and the bottom left $n \times n$ block is upper triangular.

\begin{remark}\label{rmk:resultant}
Note that there is an important, yet subtle, difference between the Sylvester matrix formulation in \cite{SS}, which uses total degree, and \cite{PetBook} which uses the degree in $z_1$. This difference (along with the fact that we will be interested in solutions satisfying the constraint $z_2 = \ol{z_1}$) is the source of the relative advantage of one method over the other, as will be seen below. 
\end{remark}

Here, the \textit{$z_1$-Sylvester resultant} of $P$ and $Q$ is the determinant of the $z_1$-Sylvester matrix of $P$ and $Q$ and will be denoted 
\be
\mathcal{R}_{z_1}(P,Q) = \det \mathcal{S}_{z_1} (P,Q)
\ee
or as $\mathcal{R}$ where the context is unambiguous. It is itself a polynomial in $z_2$ whose degree is bounded by $nt + sm$. \cite[pg 437-438]{PetBook}. 

We say that $P$ and $Q$ are \textit{coprime} if their greatest common factor has degree zero, i.e. $P=CP_1$ and $Q=CQ_2 \implies \deg C = 0$. The following is shown in Proposition 1 of \cite[pg 163]{CLO-IVA}.

\begin{lemma}\label{lem:ZeroRes}
Let $P$ and $Q$ be bivariate polynomials in $z_1$ and $z_2$. Then $\mathcal{R}$ is identically zero if and only if $P$ and $Q$ are not coprime.
\end{lemma}


The next result follows from \cite[Thm. 11.10]{PetBook} and the lemma above.

\begin{thm}[The Resultant Theorem]\label{thm:Res}
Let $P$ and $Q$ be complex polynomials as in \eqref{eqn:bipolyP} and \eqref{eqn:bipolyQ}. Then $z_2$ satisfies $\mathcal{R}(z_2) = 0$ if and only if either $a_n(z_2) = d_s(z_2) = 0$ or there exists some $z_1$ such that $P(z_1,z_2) = Q(z_1,z_2) = 0$. Hence, if $P$ and $Q$ are coprime then the number of values of $z_2$ that appear in the set of common zeros $(z_1,z_2)$ of both $P$ and $Q$ is bounded from above by $nt + ms$.
\end{thm}


The Resultant Theorem is useful when paired with the additional constraint that the second variable is the conjugate of the first (as will be considered in the next subsection).
We note in passing that in the absence of such a constraint, Theorem \ref{thm:Res} provides no advantage over the Bezout bound.
Indeed, the Bezout bound \cite{sheil-small_2002} on the number of common zeros of $P$ and $Q$ as defined above is 
\be
\#\{(z_1,z_2) \;\: | \:\; P(z_1,z_2) = 0 = Q(z_1,z_2) \} \leq MN.
\ee
The Resultant Theorem, on the other hand, gives
\be
\# \{ z_2 \;\:|\;\: P(z_1, z_2)=0=Q(z_1,z_2) \;\text{for some}\; z_1\}\leq  nt+ms
\ee
or symmetrically: 
\be
\# \{ z_1 \;\:|\;\: P(z_1, z_2)=0=Q(z_1,z_2) \;\text{for some}\; z_2\}\leq sm+tn.
\ee
But since each $z_1$ solution could, a priori, have the full set of $z_2$ solutions as partners, the above two estimates together produce the rather poor estimate 
\be
\# \{ (z_1,z_2) \;\:|\;\: P(z_1, z_2)=0=Q(z_1,z_2)\}\leq (nt+ms)^2,
\ee
which is quite a bit larger than the Bezout bound. 


\subsection{Polynomials in $z$ and $\overline{z}$}\label{sec:polyanalytic}
As in the statement of Theorem \ref{thm:genReduced}, we now consider the additional restriction that $z_2 = \overline{z_1}$, so that $P(z_1,z_2)=P(z,\ol{z})$ is a polyanalytic polynomial, and we also restrict to the case that $Q(z,\ol{z}) = \ol{P(z,\ol{z})}$.


Formally treating $z$ and $\ol{z}$ as separate variables we have $\deg_z P = n = \deg_{\overline{z}} Q$ and $\deg_{\overline{z}} P = m = \deg_z Q$. Let $N=\deg P=\deg Q$. Following \cite[Ch 1.2.4]{sheil-small_2002}, the Bezout bound is then
$$
\#\{(z,\overline{z}) \;\: | \:\; P(z,\overline{z}) = 0 \} \leq N^2. 
$$
Note that $N\leq n+m$ and hence $N^2\leq(n+m)^2$. The corresponding bound provided by Theorem \ref{thm:Res} is
$$
\# \{ z_1 \;\:|\;\: P(z_1, z_2)=0 \;\text{for some}\; z_2 \}\leq m^2+n^2.
$$
But the restriction $z_2 = \overline{z_1}$ means in this case each solution matches with only one partner. In other words:
$$
\# \{ (z,\overline{z}) \;\:|\;\: P(z, \overline{z})=0\}\leq m^2+n^2.
$$
This proves the following lemma:

\begin{lemma}[The Polyanalytic Resultant Lemma]\label{lem:PRL}
Let $P$ be a polyanalytic polynomial of the form as in equation \eqref{eqn:bipolyP}. Assume $P(z,\ol{z})$ and $Q(z,\ol{z}) = \ol{P(z,\ol{z})}$, viewed as bivariate polynomials in $z, \ol{z}$, are coprime. Then the number of zeros of $P(z,\ol{z})$ is bounded above by $n^2 + m^2$.
\end{lemma}

Let us refer to the bound $n^2+m^2$ provided by this lemma as the \emph{Resultant Bound}. Trivially, the resultant bound improves the Bezout bound $N^2$ if and only if
\be\label{eq:P<B}
n^2+m^2 < N^2.
\ee
This necessary and sufficient condition is depicted in Figure \ref{fig:RvsB}.  We note that for the case $P(z,\ol{z}) = p(z) \ol{q(z)} - w$ related to logharmonic polynomials we have $N=n+m$, so that the condition \eqref{eq:P<B} holds for all $m,n \geq 1$.
On the other hand, the condition \eqref{eq:P<B} does not hold for any nontrivial harmonic polynomials, i.e., for polynomials of the form $f(z) + \ol{g(z)}$ having $\deg f, \deg g >0$.

\begin{figure}
  \centering
\includegraphics[width=.4\linewidth]{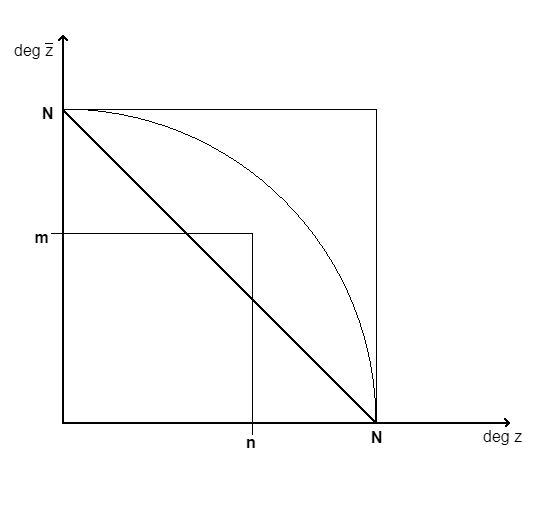}
\includegraphics[width=.4\linewidth]{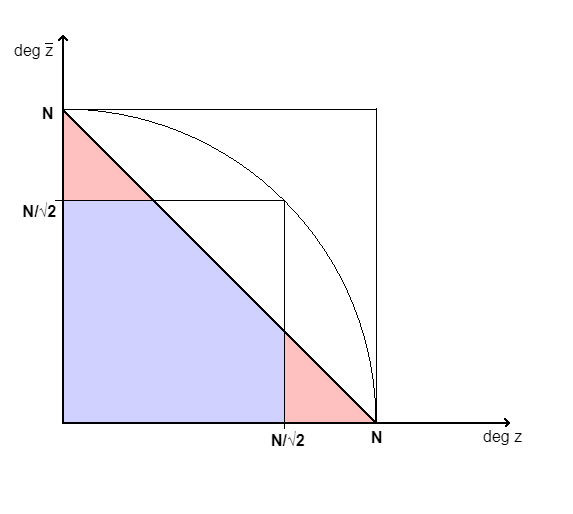}
  \caption{ Left: A depiction of the condition $n^2 + m^2 < N^2$. The monomials of a polyanalytic polynomial $P(z,\overline{z})$ of total degree $N$ correspond to integer lattice points that lie on and within the right triangle. Note that the lattice point $(n,m)$, which is constructed from the maximum degrees in $z$ and $\overline{z}$, need not correspond to an individual monomial. Right: A depiction of a simple sufficient condition $n,m < N/\sqrt{2}$. If the monomials are limited to the pentagonal region, the Resultant Bound is less than the Bezout bound.}
\label{fig:RvsB}
\end{figure}

\subsection{Coprimacy of $P$ and $\overline{P}$ and proof of Theorem \ref{thm:genReduced}}
Let $p$ and $q$ be complex univariate polynomials of degree $n$ and $m$ respectively.
We write these as 
\be
p(z)=\sum_{k=0}^n p_k z^k \;\;\;\text{and} \;\;\; q(z)=\sum_{k=0}^m q_k z^k
\ee
where $p_k,q_k\in \CC$ and $p_n,q_m \neq 0$.


As in the statement of Theorem \ref{thm:genReduced}, let us denote
\be\label{eqn:HenSimp}
P(z,\overline{z}) = p(z)\overline{q(z)} - 1,
\ee
and as in the statement of Lemma \ref{lem:PRL} we denote
\be
Q(z,\overline{z}) = \overline{P(z,\overline{z})} = \overline{p(z)}q(z) - 1.
\ee

In view of Lemma \ref{lem:PRL}, all that remains to prove Theorem \ref{thm:genReduced} is to establish the following lemma ensuring that $P$ and $Q$ are coprime if $p$ is not a constant multiple of $q$.

\begin{lemma}\label{lem:coprime1}
Given two complex polynomials $p$ and $q$, let $P(z,\overline{z}) = p(z)\overline{q(z)} - 1 $ and $Q(z,\ol{z}) = \overline{p(z)}q(z) - 1$. Then either $p(z)= cq(z)$ for some complex constant $c$, or $P$ and $Q$ are coprime.    
\end{lemma}

The proof of the lemma will use the characterization of coprimacy in terms of resultants provided by Lemma \ref{lem:ZeroRes}.
Note that $P$ and $Q$ may be written
\be
P(z, \overline{z}) = - 1 + \sum_{k=1}^n p_k \overline{q(z)} z^k, \quad \quad
Q(z, \overline{z}) = q_0\overline{p(z)} - 1 + \sum_{k=1}^m q_k\overline{p(z)} z^k,
\ee
where we have assumed, (by precomposing with a translation), $p_0 = 0$.

Having written $P,Q$ as polynomials in $z$ with coefficients depending on $\ol{z}$, recall that the resultant $\mathcal{R}_{z}(P,Q)$ of $P,Q$ is expressed in terms of those coefficients 
as the determinant of the Sylvester matrix \eqref{eqn:Sylvester}:
\be  \label{eqn:BigSyl}
\left[\arraycolsep=3pt\def\arraystretch{1.4}
\begin{array}{cccc|cccc}
p_n\overline{q(z)} & p_{n-1} \overline{q(z)} & \cdots & p_1\overline{q(z)} & -1 & 0 & \cdots & 0 \\
0 & p_n\overline{q(z)} & p_{n-1} \overline{q(z)} & \cdots & p_1\overline{q(z)} & -1 & \cdots & 0 \\
\vdots  &  & \ddots &  & &  & \ddots \\ 
0 & \cdots & 0 & p_n\overline{q(z)} & & \cdots  &  & -1 \\ \hline 
q_m\overline{p(z)} & \cdots & q_1 \ol{p(z)} & q_0\overline{p(z)}-1 & 0 & 0 &\cdots & 0 \\
0 & q_m\overline{p(z)} &  \cdots & q_1 \ol{p(z)} & q_0\overline{p(z)}-1 & 0 & \cdots & 0 \\
\vdots &  & \ddots &  & &  \\
0 & \cdots & 0 & q_m\overline{p(z)} & q_{m-1}\overline{p(z)} & & \cdots & q_0\overline{p(z)}-1 \\
\end{array}
\right].
\ee

To prove Lemma \ref{lem:coprime1}, we first establish the following intermediate step.
\begin{lemma}\label{lem:coprime2}
With $P,Q$ as in the statement of Lemma \ref{lem:coprime1}, if $P$ and $Q$ are not coprime then
\be
\{z\in \CC  :  p(z) = 0\}=\{z\in \CC  :  q(z) = 0\}.
\ee
\end{lemma}
\begin{proof}[Proof of Lemma \ref{lem:coprime2}]
Proving the contrapositive, we assume without loss of generality that there exists $z_0 \in \CC$ such that $p(z_0) \neq 0$ while $q(z_0) = 0 $. We then show that the resultant $\mathcal{R}_{z}(P,Q)$ is not identically zero and hence $P$ and $Q$ are coprime by Lemma \ref{lem:ZeroRes}. Substituting $z_0$ into \eqref{eqn:BigSyl} and evaluating, we have a matrix of the form

\be
\left [
\begin{array}{cccc|cccc}
0 & 0 & \cdots & 0 & -1 & 0 & \cdots & 0 \\
0 & 0 & \cdots & 0 & 0 & -1 & \cdots & 0 \\
&\vdots& & & &  & \ddots\\
0 & 0 & \cdots & 0 & 0 & 0& \cdots & -1 \\
\hline 
 & & & & & & &\\
q_m \overline{p(z_0)} & q_{m-1} \overline{p(z_0)}  & \cdots &  & &   & \cdots &  \\
0 & q_{m} \overline{p(z_0)}  & \cdots &  & &   & \cdots &  \\
 \vdots &  & \ddots &  &   & & \cdots  \\
0 & \cdots & 0 & q_m \overline{p(z_0)} & & & \cdots \\
\end{array}
\right ]
\ee
whose determinant is $(-1)^m (q_m\overline{p(z_0)})^n\neq 0.$ Since $\mathcal{R}_{z}(P,Q)$ is not identically zero, it follows from Lemma \ref{lem:ZeroRes} that $P$ and $Q$ are coprime.
\end{proof}

We now prove Lemma \ref{lem:coprime1} (and hence conclude Theorem \ref{thm:genReduced} as explained above).

\begin{proof}[Proof of Lemma \ref{lem:coprime1}]
Suppose that $p$ is not a constant multiple of $q$ and, working toward a contradiction, suppose that $P$ and $Q$ are not coprime. From Lemma \ref{lem:coprime2} we have that $\{p=0\} = \{q=0\}$, but since $p$ is not a constant multiple of $q$ there exists $z_0 \in \{z: p(z) = q(z) = 0\}$ where $p$ and $q$ have distinct order of vanishing, i.e., as $z\to z_0$ \be 
\begin{split}
p(z) &= \alpha (z-z_0)^k + \mathcal{O}((z-z_0)^{(k+1)}), \;\;\text{and}\\ 
q(z) &= \beta (z-z_0)^\ell  + \mathcal{O}((z-z_0)^{(\ell +1 )}),
\end{split} 
\ee
with $k\neq \ell$.  Assume $k > \ell$ without loss of generality. We then have 
\be
    \frac{\mathcal{R}_z(P,Q)}{(\ol{z}-\ol{z_0})^{kn}} = \det
\left [
\begin{array}{cccc|cccc}
 &  &  &  & -1 & 0 & \cdots & 0 \\
 & \mathcal{A} &  &  &  & -1 & \cdots & 0 \\
 &  &  &  &  & &  \ddots \\
 &  &  &  & \mathcal{B} &  &  & -1 \\
\hline 
& & & & & & & \\
\frac{q_m \overline{p(z)}}{(\ol{z}-\ol{z_0})^k} & &  & \mathcal{C} & &   &  &  \\
0 & \frac{q_m \overline{p(z)}}{(\ol{z}-\ol{z_0})^k} &  &  & &   & \mathcal{D} &  \\
0& 0 & \cdots & \frac{q_m \overline{p(z)}}{(\ol{z}-\ol{z_0})^k} & & & \\
\end{array}
\right ]
\ee
where we have multiplied each of the first $n$ columns by $1/(\ol{z}-\ol{z_0})^k$. Note that in submatrix $\mathcal{A}$ the term $\overline{q(z)}/(\ol{z}-\ol{z_0})^k$ appears as a factor in each entry. Using row operations, each entry in triangular component $\mathcal{C}$ may be eliminated (utilizing the nonzero diagonal entries below component $\mathcal{C}$) while only disturbing entries in submatrix $\mathcal{D}$. Letting $z\to z_0$ and using 
\be
\lim_{z\to z_0} \frac{p(z)}{(z-z_0)^k} = \alpha \;\;\; \text{and} \;\;\; \lim_{z\to z_0} \frac{q(z)}{(z-z_0)^k} = 0
\ee
we obtain
\be
   \lim_{z\to z_0} \frac{\mathcal{R}_z(P,Q)}{(\ol{z}-\ol{z_0})^{kn}} = \det
\left [
\begin{array}{cccc|cccc}
0 & 0 & \cdots & 0 & -1 & 0 & \cdots & 0 \\
0 & 0 & \cdots & 0 &  & -1 & \cdots & 0 \\
\vdots &  &  & \vdots &   & & \ddots & \\
0 & 0 & \cdots & 0 & \mathcal{B} &  &  & -1 \\
\hline 
& & & & & & & \\
q_m \ol{\alpha} & 0 & \cdots & 0 &  &   &  &  \\
0 & q_m \ol{\alpha} & \cdots & 0 &  &   & \mathcal{D}^* &  \\
 &  & \ddots &  &   &  \\
0& 0 & \cdots & q_m \ol{\alpha} & & & \\
\end{array}
\right ]
=(-1)^m(q_m \ol{\alpha})^n \neq 0,
\ee
which contradicts the assumption that  $P$ and $Q$ are not coprime and concludes the proof of the lemma.
\end{proof}




\subsection*{Acknowledgments}
The authors are grateful to Walter Bergweiler for producing the images of the basins of attraction shown in Figure \ref{fig:Basins}.  We also thank Alexandre Eremenko for helpful comments.  The first named author acknowledges support from the Simons Foundation (grant 513381), and the second-named author also acknowledges support from the Simons Foundation (grant 712397).

\bibliographystyle{abbrv}
\bibliography{Logharmonic}

\end{document}